\documentclass[11pt,reqno]{amsart}
\usepackage{geometry}
\usepackage{stmaryrd}
\usepackage{amsfonts,amsmath,amssymb}
\usepackage{mathabx}
\usepackage{mathrsfs}
\usepackage[latin1]{inputenc}
\usepackage[T1]{fontenc}
\usepackage[usenames,dvipsnames]{pstricks}
\newtheorem{theorem}{Theorem}[section]

\newtheorem{remark}[theorem]{Remark}

\newtheorem{lemma}[theorem]{Lemma}

\newtheorem{proposition}[theorem]{Proposition}

\newcommand{\R}{\mathbb{R}}
\newcommand{\C}{\mathbb{C}}

\newcommand{\cQ}{\mathcal{Q}}

\newcommand{\im}{\mathrm{Im}}
\newcommand{\norm}[1]{\|#1\|}
\newcommand{\abs}[1]{|#1|}

\usepackage{graphicx}

\usepackage{mathrsfs}
\usepackage[latin1]{inputenc}
\usepackage[T1]{fontenc}
\usepackage{color,xspace}
\usepackage[usenames,dvipsnames]{pstricks}
\usepackage{times}
\allowdisplaybreaks
\usepackage{subcaption}
\usepackage{hyperref}
\hypersetup{
    colorlinks=true,
    linkcolor=blue,
    filecolor=blue,
    urlcolor=cyan,
    pdftitle={Overleaf Example},
    pdfpagemode=A4
    }

\begin{document}
\title[Identification of unbounded electric potentials through asymptotic boundary spectral data]
{Identification of unbounded electric potentials through asymptotic boundary spectral data}
%
%
\address{Université de Tunis El Manar, \'Ecole Nationale d'Ingénieurs de Tunis, ENIT-LAMSIN, B.P. 37, 1002 Tunis, Tunisia}
\email{mourad.bellassoued@enit.utm.tn}
%
\author[M. Bellassoued,  Y. Kian, Y. Mannoubi, \'E. Soccorsi]{Mourad Bellassoued,  Yavar Kian, Yosra Mannoubi, \'Eric Soccorsi}
\address{Université de Tunis El Manar, \'Ecole Nationale d'Ingénieurs de Tunis, ENIT-LAMSIN, B.P. 37, 1002 Tunis, Tunisia}  
\email{yosra.mannoubi@enit.utm.tn} 

\address{Aix-Marseille Univ., Université de Toulon, CNRS, CPT, Marseille, France}

\email{yavar.kian@univ-amu.fr}
 
\address{Aix-Marseille Univ., Université de Toulon, CNRS, CPT, Marseille, France}

\email{eric.soccorsi@univ-amu.fr}


\maketitle               

\begin{abstract}             
We prove that the real-valued electric potential $q \in L^{\max(2,3 n \slash 5)}(\Omega)$ of the Dirichlet Laplacian
$-\Delta +q$ acting in a bounded domain $\Omega \subset \mathbb{R}^n$, $n \ge 3$, is uniquely determined by the asymptotics of the eigenpairs formed by the eigenvalues and the boundary observation of the normal derivative of the eigenfunctions. 
\end{abstract}

\section{Introduction}
\subsection{Statement of the main result}
Let $\Omega \subset \mathbb{R}^n$, $n \geq 3$, be a bounded domain with $\mathcal C^2$ boundary $\Gamma=\partial\Omega$. 
Let $q \in \cQ_c := \left\{ q \in L^{\max(2,3n \slash 5)}(\Omega,\R)\ \mbox{s. t.}\ q(x) \geq -c,\ x \in \Omega \right\}$, where $c$ is an {\it a priori} fixed positive constant. We consider the perturbed Dirichlet Laplacian $A_q=-\Delta +q$ in $L^2(\Omega)$, i.e., the self-adjoint operator generated in $L^2(\Omega)$ by the 
closed Hermitian form
\begin{equation}
\label{def-aq}
a_q(u,v)= \int_\Omega \left( \nabla u \cdot \nabla \overline{v} + q u \overline{v} \right) dx,\ u, v \in D(a_q):=H_0^1(\Omega),
\end{equation}
see Appendix \ref{app-A} . 
Since the embedding $H_0^1(\Omega) \subset L^2(\Omega)$ is compact, the operator $A_q$ has a compact resolvent and there exist a sequence of eigenfunctions $\phi_k \in D(A_q)=\{ u \in H_0^1(\Omega),\ (-\Delta+q) u \in L^2(\Omega) \}$ which form an orthonormal basis of $L^2(\Omega)$, and a sequence of eigenvalues
$$-\infty<\lambda_{1}\leq \lambda_{2}\leq \ldots \leq \lambda_{k} \leq \lambda_{k+1} \leq \ldots$$
satisfying $\lim_{k \to +\infty} \lambda_k=+\infty$ and
\begin{equation}
\label{egn}
A_q \phi_k=\lambda_k \phi_k,\ k \geq 1.
\end{equation}
For $k \geq 1$, we set $\psi_k:=(\partial_\nu \phi_k)_{| \Gamma}$, where $\nu$ denotes the outward unit vector to $\Gamma$.

In the present paper we examine the inverse spectral problem of knowing whether knowledge of the asymptotic behavior (with respect to $k$) of the boundary spectral data $\{(\lambda_k,\psi_k),\ k \geq 1\}$ uniquely determines $q$.

The study of inverse spectral problems goes back at least to 1929 and Ambarsumian's pioneer article \cite{A}. Later on, Borg \cite{Bo}, Levinson \cite{L}, and Gel'fand and Levitan  \cite{GL} proved that knowledge of the spectrum and additional spectral data uniquely determines the electric potential of one dimensional Schr\"odinger operators. 

Gel'fand and Levitan's result was adapted to the multi-dimensional case by Nachman, Sylvester and Uhlmann in \cite{NSU}, where the potential was identified through the eigenpairs formed by the eigenvalues and the boundary measurement of the normal derivative of the eigenfunctions of the Dirichlet Laplacian. While full knowledge of the boundary spectral data was requested by \cite{NSU}, Isozaki retrieved the potential in \cite{I} when finitely many eigenpairs remain unknown. Further downsizing the data, Choulli and Stefanov established in \cite{CS} that asymptotic knowledge of the boundary spectral data is enough to recover the potential. This result was improved by \cite{KKS,S} upon weakening the condition imposed on the asymptotic spectral data. The analysis carried out in \cite{KKS,S} was extended to magnetic Schr\"odinger operators in \cite{Ki1} and to Riemannian manifolds in \cite{BCFKS}. In the above mentioned articles, the measurement of the Neumann data is performed on the entire boundary of the domain. We refer the reader to \cite{BM} for a Borg-Levinson theorem with partial boundary measurement of the normal derivative of the eigenfunctions.  

The stability issue in the inverse problem of determining the electric potential from boundary spectral data was first treated by Alessandrini and Sylvester in \cite{AS}. We refer the reader to \cite{Ch,BCY,BD,CK1,CS,KKS} for the analysis of this problem under various conditions on both the unknown potential and the boundary spectral data.  

All the above mentioned results were obtained for Schr\"odinger operators with regular coefficients. Actually, there is only a small number of mathematical papers studying inverse spectral problems with singular coefficients. In \cite{PS}, P\"avarinta and Serov retrieved unknown potentials in $L^p(\Omega, \mathbb{R})$, $p >n/2$, from the full boundary spectral data. More recently, in \cite{P}, Pohjola showed unique determination of the electric potential in $L^{n/2}(\Omega,\mathbb{R})$ from either full boundary spectral data when $n=3$ or incomplete boundary spectral data when $n \geq 4$, and of an unknown potential in $L^{p}(\Omega,\mathbb{R})$ with $p >n/2$ and $n=3$, from incomplete boundary spectral data. As far as we know, there is no result available in the mathematical literature, dealing with the identification of a singular potential by asymptotic knowledge of boundary spectral data similar to the ones used in \cite{BCFKS,CS,KKS,Ki1}.


\subsection{Main result}

The main result of this article is as follows.
\begin{theorem}
\label{t1}
Let $q_j\in \cQ_c$, $j=1,2$, for some constant $c>0$. Denote by $\{(\lambda_{j,k},\psi_{j,k}),\ k \geq 1 \}$ the boundary spectral data of the operator $A_{q_j}$ and assume that
\begin{equation}
\label{t1a}   \lim_{k \rightarrow \infty} (\lambda_{1,k}-\lambda_{2,k})=0 \; \;\text{and}\;\; \sum_{k=1}^{+ \infty} \norm{\psi_{1,k}-\psi_{2,k}}_{L^2(\Gamma)}^2 <\infty.
\end{equation}
Then, we have $q_1=q_2$.
\end{theorem}

To the best of our knowledge, Theorem \ref{t1} contains the first Borg-Levinson identification result of a singular potential by asymptotic knowledge of the boundary spectral data. Notice that \eqref{t1a} is the exact same condition requested by \cite{KKS,S} on the boundary spectral data, in the identification of a bounded electric potential. It turns out that this condition still applies to the determination of a singular potential lying in $L^{3n/5}(\Omega)$ when $n \ge 4$ or in $L^2(\Omega)$ when $n=3$. Actually, Theorem \ref{t1} goes one step further than \cite{P} in downsizing the boundary spectral used for recovering a singular potential, but this is at the expense of greater regularity for the admissible unknown coefficient ($L^2(\Omega)$ instead of $L^p(\Omega)$, $p>3 \slash 2$ when $n=3$ and $L^{3n \slash 5}(\Omega)$ instead of $L^{n\slash 2}(\Omega)$ when $n \ge 4$).  

There are two main ingredients in the proof of Theorem \ref{t1}. The first one is a suitable $L^2(\Gamma)$-estimate of the normal derivative of any $H^1(\Omega)$-solution to the Laplace equation. It is given in Proposition \ref{pr1} and generalizes a classical elliptic regularity result to the case
of $L^{3n \slash 5}(\Omega)$-potentials. The second one is an adaptation of Isozaki's representation formula, presented in \cite{I}, to the framework of singular potentials lying in $\cQ_c$.

\subsection{Outline}
This paper is organized as follows. In Section \ref{sec-pre}, we establish several technical results which are useful for the proof of Theorem \ref{t1}. In Section \ref{sec-Isozaki} we adapt the celebrated Isozaki's representation formula to the case of singular potentials. Finally, Section \ref{sec-proof} contains to the proof of Theorem \ref{t1} and the Appendix is devoted to the definition of $A_q$.

\section{Preliminaries}
\label{sec-pre}

Let $q\in L^{3n/5}(\Omega)$ be real-valued. For $\lambda \in \mathbb{C}$ and $f \in H^{3/2}(\Gamma)$, we consider the following boundary value problem (BVP)
 \begin{equation}
 \label{1}  
 \left\{  
\begin{array}{ll}
 (-\Delta +q-\lambda)u=0 &\text{in}\; \Omega\\
   u=f &\text{on}\;\Gamma. \\          
 \end{array}                     
\right. 
\end{equation}               
As will appear in the remaining part of this article, taking $f$ in $H^{3/2}(\Gamma)$ is enough for the purpose of this work and we point out that unlike in \cite{P}, where more exotic Dirichlet data were considered, there is no need to go this way in the present paper.

With reference to \cite[Lemma 2.3 and Corollary 2.4]{P}, there exists $\lambda_0>0$ such that for all $\lambda \in\mathbb C\setminus(- \lambda_0,+\infty)$, the BVP \eqref{1} admits a unique solution $u \in W^{2,p}(\Omega)$, where $p=2n/(n+2)$. Moreover, $u$ satisfies the estimate
\begin{equation}\label{u}
\| u \|_{W^{2,p}(\Omega)} \leq C_\lambda \|f\|_{H^{3/2}(\Gamma)}
\end{equation}
for some positive constant $C_\lambda$ which depends on $\lambda$.

In this section we 
aim to study the influence of either the potential $q$ or the spectral parameter $\lambda$, 
on the solution $u$ to \eqref{1}. More precisely, bearing in mind that the trace operator 
\begin{equation}
\label{123}
\begin{array}{ccccc}
\tau_1 & : & W^{2,p}(\Omega) & \longrightarrow & W^{1-1/p,p}(\Gamma) \\
 & & u & \longmapsto & {\partial_\nu u}, 
\end{array}
\end{equation}
is continuous, we shall first examine the dependence of $\partial_\nu u$ with respect to the electric potential when $\lambda$ is sent to $-\infty$, and, in a second time, with respect to the spectral parameter when the potential $q$ is fixed.


\subsection{Influence of the potential in the asymptotic regime $\lambda \to -\infty$}
The result that we have in mind is inspired by \cite[Lemma 2.5]{KKS} and \cite[Lemma 2.3]{S}. For any two potentials  
$q_j \in \cQ_c$, $j=1,2$, where $c>0$ is fixed, it indicates that the two solutions to \eqref{1} associated with either $q = q_1$ or $q = q_2$ are close as $\lambda$ goes to $-\infty$.

\begin{lemma}
\label{lemma 2.2}
Put $p:=2n/(n+2)$, let $f\in  H^{3/2}(\Gamma)$ and pick $q_j \in \cQ_c$, $j=1,2$, for some $c>0$. For $\lambda \in \R \setminus (\mathrm{Sp}(A_{q_1})\cup \mathrm{Sp}(A_{q_2}))$, denote by $u_{j,\lambda}$ the solution to the BVP \eqref{1} with $q=q_j$. Then, we have
\begin{equation}
\label{ll1a}
\lim_{\lambda \to -\infty}\| \partial_\nu u_{1,\lambda}- \partial_\nu u_{2,\lambda} \|_{L^p(\Gamma)} = 0.
\end{equation}
\end{lemma}
\begin{proof}
Since the function $u_\lambda = u_{1,\lambda}-u_{2,\lambda}$ solves
$$ 
 \left\{  
\begin{array}{ll}
 (-\Delta +q_1-\lambda)u_\lambda=(q_2-q_1)u_{2,\lambda} &\text{in}\; \Omega\\
   u_\lambda=0 &\text{on}\;\Gamma,\\          
 \end{array}                     
\right. $$ 
we deduce from \cite[Theorem 2.3.3.6]{Gr} that 
\begin{equation}
\label{a1}
\| u_\lambda \|_{W^{2,p}(\Omega)} \leq C \| (q_2-q_1) u_{2,\lambda} \|_{L^p(\Omega)}. 
\end{equation}
Here and in the remaining part of this proof, $C$ denotes a positive constant that may change from line to line, but which is always independent of $\lambda$.

Further, from the continuity of the trace operator $\tau_1$ defined in \eqref{123}, we get for all 
$\epsilon \in\left(0,1-1/p\right)$ that
\begin{align}\label{2.8}
\| \partial_\nu u_\lambda \|_{W^{1-\frac{1}{p}-\epsilon,p}(\Gamma)} &\leq C \| u_\lambda \|_{W^{2-\epsilon,p}(\Omega)}\cr
&\leq C \| u_\lambda \|^{\frac{\epsilon}{2}}_{L^p(\Omega)} \| u_\lambda \|_{W^{2,p}(\Omega)}^{1-\frac{\epsilon}{2}},
\end{align}
upon interpolating between $L^p(\Omega)$ and $W^{2,p}(\Omega)$.
On the other hand, we have
\begin{align}\label{u_2}
 \| u_{2,\lambda} \|_{L^{\frac{2n}{n-2}}(\Omega)} \leq  C\| f \|_{H^{\frac{3}{2}}(\Gamma)}
\end{align}
from \cite[Lemma 3.2.]{P}.

Moreover, we have
$\| u_\lambda \|_{L^p(\Omega)} \leq \frac{C}{| \lambda |} \| (q_2 -q_1) u_{2,\lambda}\|_{L^{p}(\Omega)}$ by virtue of \cite[Lemma 3.1]{P}, and hence
\begin{eqnarray*}
\| u_\lambda \|_{L^p(\Omega)} 
&\leq & \frac{C}{| \lambda |} \| q_2 -q_1 \|_{L^{\frac{n}{2}}(\Omega)} \| u_{2,\lambda} \|_{L^{\frac{2n}{n-2}}(\Omega)} \\
&\leq & \frac{C}{| \lambda |} \| q_2 -q_1 \|_{L^{\frac{3n}{5}}(\Omega)} \| u_{2,\lambda} \|_{L^{\frac{2n}{n-2}}(\Omega)}
\end{eqnarray*}
from \cite[Theorem 2.3.1.5]{Gr} and the H\"older inequality. 
From this and \eqref{u_2}, it then follows that
\begin{equation}
\label{a2} 
\| u_\lambda \|_{L^p(\Omega)} \leq  \frac{C}{| \lambda |}.
\end{equation}
Similarly, using \eqref{a1} we get that
\begin{eqnarray*}
\| u_{\lambda} \|_{W^{2,p}(\Omega)} &\leq & C  \| q_2 -q_1 \|_{L^{\frac{n}{2}}(\Omega)} \| u_{2,\lambda} \|_{L^{\frac{2n}{n-2}}(\Omega)}\\
&\leq & C \| q_2 -q_1 \|_{L^{\frac{3n}{5}}(\Omega)} \| u_{2,\lambda} \|_{L^{\frac{2n}{n-2}}(\Omega)} ,
\end{eqnarray*}
and consequently $\| u_{\lambda} \|_{W^{2,p}(\Omega)} \leq C$ by \eqref{u_2}.
Putting this together with \eqref{2.8} and \eqref{a2}, we find 
that
$$
\| \partial_\nu u_\lambda \|_{W^{1-\frac{1}{p}-\epsilon,p}(\Gamma)}\leq \frac{C}{| \lambda |^{\frac{\epsilon}{2}}}
$$
for all $\epsilon \in\left(0,1-1/p \right)$, which immediately entails that
\begin{equation}
\label{129}
\lim_{\lambda \rightarrow -\infty} \| \partial_\nu u_\lambda \|_{W^{1-\frac{1}{p}-\epsilon,p}(\Gamma)} = 0.
\end{equation}
Since $W^{1-1/p-\epsilon,p}(\Gamma)$ is continuously embedded into $L^p(\Gamma)$ (see e.g., \cite[Sections 1.3.1 and 1.3.3]{Gr}), the desired result follows readily from \eqref{129}.
\end{proof}

Lemma \ref{lemma 2.2} establishes that, in some sense, the influence of the potential is dimmed when the spectral parameter $\lambda$ is sent to $-\infty$. Having seen this, we turn now to the representation of the solution to \eqref{1} 
in terms of $\lambda$ and the boundary spectral data $\{ (\lambda_k,\psi_k),\ k \geq 1\}$ of the operator $A_q$.

\subsection{A representation formula}

Let $u_\mu$ denote the solution to \eqref{1} associated with spectral parameter $\lambda=\mu \in \C \setminus \text{Sp}(A_q)$. 
In this section, we aim to express $\partial_\nu (u_\lambda-u_\mu)$ in terms of $\lambda$, $\mu$ and the boundary spectral data of $A_q$. Prior to doing that, we will establish that all the $\psi_k$'s lie in $L^2(\Gamma)$. This technical result, which is a byproduct of 
Proposition \ref{pr1} below, is a key-point in the derivation of the representation formula of 
$\partial_\nu (u_\lambda-u_\mu)$ given in Lemma \ref{l5}.

Let $F\in L^2(\Omega)$. For $q\in L^\infty(\Omega)$, we recall from the classical elliptic regularity theory (see e.g.) that the BVP
\begin{equation} 
\label{22}
 \left\{
\begin{array}{ll}
( -\Delta +q)u=F  &\text{in} \;\Omega\\
    u=0  &\text{on} \;\Gamma
 \end{array}
\right.
\end{equation}
admits a unique solution $u \in H^2(\Omega)$ satisfying
$$\norm{u}_{H^2(\Omega)}\leq C(\norm{u}_{L^2(\Omega)}+\norm{F}_{L^2(\Omega)}) $$
for some positive constant $C$ which depends only on $\Omega$ and $\norm{q}_{L^\infty(\Omega)}$.
This yields $\partial_\nu u\in L^2(\Gamma)$ and the following estimate
\begin{equation} 
\label{l3a}
\norm{\partial_\nu  u}_{L^2(\Gamma)} \leq C(\norm{u}_{L^2(\Omega)}+\norm{F}_{L^2(\Omega)}),
\end{equation}
where $C$ is another positive constant depending only on $\Omega$ and $\norm{q}_{L^\infty(\Omega)}$. 

However, if $q$ is unbounded, in general there is no such thing as a $H^2(\Omega)$-solution to \eqref{22}. Hence the $L^2(\Gamma)$-regularity of the normal derivative $\partial_\nu u$ and consequently the energy estimate \eqref{l3a} are no longer guaranteed by the standard theory of elliptic PDEs. Nevertheless, we shall establish in the following proposition that these two properties remain valid provided $q$ is taken in $L^{3n \slash 5}(\Omega)$.

\begin{proposition}
\label{pr1}
Let $F\in L^2(\Omega)$ and let $q \in \cQ_c$ for some $c>0$. Let $u\in H^1(\Omega)$ be a solution to \eqref{22}. Then, we have $\partial_\nu  u\in L^2(\Gamma)$ and the estimate \eqref{l3a} holds for some positive constant $C$ depending only on $\Omega$ and $q$.
\end{proposition}
\begin{proof} 
Without loss of generality we assume that $F$ is real-valued in such a way that the solution $u$ to \eqref{22} is real-valued as well.

Since $q+c \ge 0$ by assumption, there exists a sequence $(q_\ell)_{\ell \ge 1} \in \mathcal C^\infty(\overline{\Omega})$ of non-negative functions, such that 
\begin{equation}
\label{eq-s0}
\lim_{\ell \to +\infty} \norm{q_\ell-(q+c)}_{L^{3n/5}(\Omega)} =0,
\end{equation}
and for each $\ell \ge 1$, we consider the solution $u_\ell \in H^2(\Omega)\cap H^1_0(\Omega)$ to the following BVP:
\begin{equation} \label{222}
 \left\{
\begin{array}{ll}
( -\Delta +q_\ell)u_\ell=cu+F  &\text{in} \;\Omega\\
    u_\ell=0  &\text{on} \;\Gamma.\\
 \end{array}
\right.
\end{equation}
The derivation of Proposition \ref{pr1} being quite lengthy,
we split the rest of the proof into four steps. In the first one we establish that the sequence $(u_\ell)_{\ell \ge 1}$ is bounded in $H^1(\Omega)$. The second step is to prove that $(u_\ell)_{\ell \ge 1}$ converges to $u$ in the $W^{2,p}(\Omega)$-norm topology, where $p=2n/(n+2)$. In the third step we show that $(u_\ell)_{\ell \ge 1}$ is bounded in $W^{2,p_1}(\Omega)$, with $p_1=6n/(3n+4)$. Finally, the fourth step contains the end of the proof of Proposition \ref{pr1}.

\noindent{\it Step 1: $(u_\ell)_{\ell\ge 1}$ is bounded in $H^1(\Omega)$.} Let $\ell \in \mathbb N$ be fixed. We multiply the first equation of \eqref{222} by $u_\ell$ and integrate over $\Omega$. We get that
$$\int_\Omega \abs{\nabla u_\ell}^2 dx + \int_\Omega q_\ell \abs{u_\ell}^2 dx =\int_\Omega (cu +F) u_\ell dx,$$
by applying the Green formula. Adding $\int_\Omega (q+c) \abs{u_\ell}^2 dx$ on both sides of the above equality then yields
\begin{equation}
\label{eq-s1.1}
\int_\Omega \abs{\nabla u_\ell}^2 dx + \int_\Omega (q+c) \abs{u_\ell}^2 dx  =\int_\Omega G u_\ell dx- \int_\Omega (q_\ell-(q+c)) \abs{u_\ell}^2 dx,
\end{equation}
where $G=cu +F$. Since $q+c\geq 0$ in $\Omega$ and $u_\ell \in H_0^1(\Omega)$, we infer from \eqref{eq-s1.1} and the Poincar\'e inequality that
$$\norm{u_\ell}^2_{H^1(\Omega)} \leq  C_0 \left( \int_\Omega \abs{q_\ell-(q+c)} \abs{u_\ell}^2 dx +\int_\Omega \abs{G}  \abs{u_\ell} dx \right), $$
where $C_0$ is a positive constant depending only on $\Omega$.
Therefore, for all $\epsilon >0$ we get 
$$
\norm{u_\ell}^2_{H^1(\Omega)} 
\leq  C_0 \left( \norm{q_\ell-(q+c)}_{L^{\frac{3n}{5}}(\Omega)} \norm{u_\ell}^2_{L^{\frac{6n}{3n-5}}(\Omega)}+  \frac{\epsilon}{2} \norm{u_\ell}^2_{L^2(\Omega)} + \frac{\epsilon^{-1}}{2} \norm{G }_{L^2(\Omega)}^2 \right) 
$$
from H\"older's inequality,  
and hence
\begin{equation}
\label{eq-s1.2}
\norm{u_\ell}^2_{H^1(\Omega)} \leq  C_0 \left(  \norm{q_\ell-(q+c)}_{L^{3n/5}(\Omega)} \norm{u_\ell}^2_{H^1(\Omega)}+\frac{\epsilon}{2} \norm{u_\ell}^2_{L^2(\Omega)}+ \frac{\epsilon^{-1}}{2} \norm{G }_{L^2(\Omega)}^2  \right).
\end{equation}
by the Sobolev embedding theorem.

Now, with reference to \eqref{eq-s0}, we pick $\ell_0 \geq 1$ such that
$\norm{q_\ell-(q+c)}_{L^{3n/5}(\Omega)}\leq \epsilon$ for all $\ell\geq\ell_0$. In light of \eqref{eq-s1.2}, this leads to
$$\norm{u_\ell}^2_{H^1(\Omega)} \leq C_0 \left( \epsilon \norm{u_\ell}^2_{H^1(\Omega)} + \epsilon^{-1} \norm{G}_{L^2(\Omega)}^2 \right),\ \ell \geq \ell_0, $$
so by choosing $\epsilon=\frac{1}{2C_0}>0$ in this inequality, we find (upon substituting $\sqrt{2} C_0$ for $C_0$) that
$$ \norm{u_\ell}_{H^1(\Omega)} \leq C_0  \norm{G}_{L^2(\Omega)},\ \ell \ge\ell_0. $$
This together with the identity $G=c u + F$ then yields
\begin{equation} 
\label{l3b}
\norm{u_\ell}_{H^1(\Omega)} \leq C\left( \norm{u}_{L^2(\Omega)} + \norm{F}_{L^2(\Omega)} \right),\ \ell \geq 1,
\end{equation}
where, from now on, $C$ denotes a positive constant depending only on $\Omega$ and $q$, which may change from line to line.

\noindent{\it Step 2: The sequence $(u_\ell)_{\ell \geq 1}$ converges to $u$ in $W^{2,p}(\Omega)$}. For $\ell\ \geq 1$ fixed, we see that the
function $v_\ell :=u-u_\ell\in H^1(\Omega)$ solves
\begin{equation}
\label{eq-s2.1}
 \left\{
\begin{array}{ll}
 -\Delta v_\ell +(q+c)v_\ell=(q_\ell-(q+c))u_\ell  &\text{in} \;\Omega\\
    v_\ell=0  &\text{on} \;\Gamma,
 \end{array}
\right.
\end{equation}
and since $q+c\geq0$, this immediately entails that
\begin{equation}
\label{l3bb}
\norm{v_\ell}_{H^1(\Omega)} \leq C \norm{(q_\ell-(q+c))u_\ell}_{H^{-1}(\Omega)}.
\end{equation}
Next, the space $H^1_0(\Omega)$ being continuously embedded in $L^{2n/(n-2)}(\Omega)$ according to the Sobolev embedding theorem, then, by duality, the space $L^p(\Omega)$ where $p=2n/(n+2)$, is continuously embedded into $H^{-1}(\Omega)$. Thus, it follows from \eqref{l3bb} that
\begin{equation} 
\label{l3cc}
\norm{v_\ell}_{H^1(\Omega)}\leq C\norm{(q_\ell-(q+c))u_\ell}_{L^p(\Omega)}.
\end{equation}
Further, we have
\begin{equation} 
 \left\{
\begin{array}{ll}
 -\Delta v_\ell =-(q+c)v_\ell+(q_\ell-(q+c))u_\ell  &\text{in}\ \Omega\\
    v_\ell=0  &\text{on} \;\Gamma,
 \end{array}
\right.
\end{equation}
by virtue of \eqref{eq-s2.1}, and consequently 
\begin{equation} 
\label{l3dd}
\norm{v_\ell}_{W^{2,p}(\Omega)}\leq C_0 (\norm{(q+c)v_\ell}_{L^p(\Omega)}+\norm{(q_\ell-(q+c))u_\ell}_{L^p(\Omega)}),
\end{equation}
from \cite[Theorem 2.4.2.5]{Gr}. Here and in the remaining part of this proof, $C_0$ denotes a generic positive constant depending only on $\Omega$, which may change from line to line. 

Moreover, since
$$\norm{(q+c)v_\ell}_{L^p(\Omega)}\leq \norm{q+c}_{L^{\frac{n}{2}}(\Omega)}\norm{v_\ell}_{L^{\frac{2n}{n-2}}(\Omega)}\leq C\norm{v_\ell}_{H^1(\Omega)},$$
by H\"older's inequality and Sobolev embedding theorem, we deduce from \eqref{l3cc} and \eqref{l3dd} that
$$\norm{v_\ell}_{W^{2,p}(\Omega)} \leq C \norm{(q_\ell-(q+c))u_\ell}_{L^{p}(\Omega)}.$$
Thus, applying the H\"older inequality and the Sobolev embedding theorem once more, we find that
$$\norm{v_\ell}_{W^{2,p}(\Omega)} \leq C \norm{q_\ell-(q+c)}_{L^{3n/5}} \norm{u_\ell}_{H^1(\Omega)}.$$
From this, \eqref{eq-s0}, \eqref{l3b} and  \eqref{l3bb} it then follows that 
$$\lim_{\ell \to +\infty} \norm{u_\ell-u}_{W^{2,p}(\Omega)}=0.$$

\noindent{\it Step 3: The sequence $(u_\ell)_{\ell \ge 1}$ is bounded in $W^{2,p_1}(\Omega)$, where $p_1=6n/(3n+4)$}. Let us recall from \eqref{222} that for all natural number $\ell$, the function $u_\ell$ solves
\begin{equation} 
\label{eq1}
 \left\{
\begin{array}{ll}
 -\Delta u_\ell=-q_\ell u_\ell+G  &\text{in} \;\Omega\\
    u_\ell=0  &\text{on} \;\Gamma,\\
 \end{array}
\right.
\end{equation}
where $G=cu +F$. Thus, by applying \cite[Theorem 2.4.2.5]{Gr} and taking into account that $L^2(\Omega)$ is continuously embedded in $L^{p_1}(\Omega)$, we obtain that
\begin{equation} 
\label{l3c}
\norm{u_\ell}_{W^{2,p_1}(\Omega)}\leq C_0 (\norm{q_\ell u_\ell}_{L^{p_1}(\Omega)}+\norm{G}_{L^2(\Omega)}).
\end{equation}
Further, we have
$$\norm{q_\ell u_\ell}_{L^{p_1}(\Omega)}\leq \norm{q_\ell}_{L^{\frac{6n}{10}}(\Omega)}\norm{u_\ell}_{L^{\frac{6n}{3n-6}}(\Omega)}\leq C_0 \norm{q_\ell}_{L^{\frac{3n}{5}}(\Omega)}\norm{u_\ell}_{H^1(\Omega)},$$
from H\"older's inequality and the Sobolev embedding theorem, and consequently
$$\norm{q_\ell u_\ell}_{L^{p_1}(\Omega)}\leq C(\norm{u}_{L^2(\Omega)}+\norm{F}_{L^2(\Omega)}),$$
from \eqref{eq-s0} and \eqref{l3b}. Putting this with \eqref{l3c}, we find that
\begin{equation} 
\label{l3d}
\norm{u_\ell}_{W^{2,p_1}(\Omega)}\leq C(\norm{u}_{L^2(\Omega)}+\norm{F}_{L^2(\Omega)}).
\end{equation}

\noindent{\it Step 4: End of the proof}. We aim to show that for all $\ell \ge 1$, we have
\begin{equation}
\label{eq-s4.1}
\norm{\partial_\nu u_\ell}_{L^2(\Gamma)} \leq C(\norm{u}_{L^2(\Omega)}+\norm{F}_{L^2(\Omega)}).
\end{equation}
For this purpose we introduce a vector field $\gamma \in \mathcal{C}^1(\overline{\Omega},\mathbb{R}^n)$ such that $\gamma_{|_{\Gamma}}=\nu$. Then, we multiply the first line of \eqref{222} by $\gamma\cdot \nabla u_\ell$, integrate over $\Omega$, and get
\begin{equation}
\label{145}
\int_\Omega (-\Delta u_\ell) \gamma\cdot \nabla u_\ell dx + \int_\Omega q_\ell\; u_\ell \gamma\cdot \nabla u_\ell dx =\int_\Omega c\;u \gamma \cdot \nabla u_\ell dx+\int_\Omega F \gamma \cdot \nabla u_\ell dx.
\end{equation}
The first term on the left-hand side of \eqref{145} is treated by the divergence formula:
\begin{equation}
\label{eq-s4.2}
\int_\Omega (\Delta u_\ell) \gamma\cdot \nabla u_\ell dx 
= \int_\Gamma | \partial_\nu u_\ell |^2 d\sigma -\int_\Omega \nabla(\gamma\cdot \nabla u_\ell)\cdot \nabla u_\ell dx.
\end{equation}
Next, writing $\gamma=(\gamma_1,\ldots,\gamma_n)^T$, we get through direct computation that
\begin{eqnarray}
\nabla(\gamma\cdot \nabla u_\ell)\cdot \nabla u_\ell &= & \sum_{i,j= 1}^n \left( \partial_i \left( \gamma_j \partial_j u_\ell \right) \right) \partial_i u_\ell 
\nonumber \\
&= & \sum_{i,j= 1}^n \left( (\partial_i \gamma_j) \partial_j u_\ell+\gamma_j \partial_i \partial_j u_\ell \right) \partial_i u_\ell \nonumber \\
&= & \sum_{i,j=1}^n (\partial_i \gamma_j) (\partial_j u_\ell) \partial_i u_\ell +\frac{1}{2} \gamma\cdot\nabla \abs{\nabla u_\ell}^2. \label{eq-s4.3}
\end{eqnarray}
Further, taking into account that $\gamma \cdot \nu=1$ on $\Gamma$ and $\abs{\nabla u_\ell}=\abs{\partial_\nu u_\ell}$ on $\Gamma$, we have $\int_\Omega \gamma \cdot\nabla \abs{\nabla u_\ell}^2 dx = \norm{\partial_\nu u_\ell}_{L^2(\Gamma)}^2 -\int_\Omega  ( \nabla \cdot \gamma) \abs{\nabla u_\ell}^2 dx$, and \eqref{eq-s4.2}-\eqref{eq-s4.3} then yield
$$\int_\Omega \Delta u_\ell (\gamma\cdot \nabla u_\ell) dx =\frac{1}{2} \norm{\partial_\nu u_\ell}_{L^2(\Gamma)}^2 +\int_\Omega H(x) \nabla u_\ell(x) dx,$$
where
$$H(x) X =-\sum_{i,j=1}^n (\partial_i \gamma_j)(x) X_j X_i + \frac{1}{2} \left( \nabla \cdot \gamma(x) \right) \abs{X}^2,\ X=(X_1,\ldots,X_n)\in\R^n,\ x \in \Omega. $$
From this and \eqref{145} it then follows that
\begin{equation}
\label{147}
\frac{1}{2} \norm{\partial_\nu u_\ell}_{L^2(\Gamma)}^2=-\int_\Omega H(x) \nabla u_\ell(x) dx+ \int_\Omega q_\ell u_\ell\ \gamma\cdot \nabla u_\ell dx -\int_\Omega G\ \gamma \cdot \nabla u_\ell dx.
\end{equation}
By H\"older's inequality, the second term on the right hand side of \eqref{147} is bounded as
\begin{eqnarray*}
\left| \int_\Omega q_\ell\; u_\ell\ \gamma\cdot \nabla u_\ell dx \right| &\leq & \norm{\gamma}_{L^\infty(\Omega)^n}\norm{q_\ell}_{L^{\frac{6n}{10}}(\Omega)}\norm{u_\ell}_{L^{\frac{6n}{3n-8}}(\Omega)}\norm{\nabla u_\ell}_{L^{\frac{6n}{3n-2}}(\Omega)}\\
&\leq & C\norm{q_\ell}_{L^{\frac{3n}{5}}(\Omega)}\norm{u_\ell}_{L^{\frac{6n}{3n-8}}(\Omega)}\norm{ u_\ell}_{W^{1,\frac{6n}{3n-2}}(\Omega)}.\end{eqnarray*}
Thus, we have
$\left|\int_\Omega q_\ell\ u_\ell\ \gamma\cdot \nabla u_\ell dx\right| \leq C \norm{q_\ell}_{L^{\frac{3n}{5}}(\Omega)}\norm{u_\ell}_{W^{2,p_1}(\Omega)}^2$
from the Sobolev embedding theorem (see e.g. \cite[Theorem 1.4.4.1]{Gr}), and consequently 
$$\left|\int_\Omega q_\ell\; u_\ell\ \gamma\cdot \nabla u_\ell dx\right|\leq C\left(\norm{u}_{L^2(\Omega)}+\norm{F}_{L^2(\Omega)}\right)^2,$$
from \eqref{l3d}. 
Putting this together with \eqref{l3b} and \eqref{147}, we obtain that
\begin{equation}
\label{ess}
\norm{\partial_\nu u_\ell}_{L^2(\Gamma)}^2
\leq C\left(\norm{u}_{L^2(\Omega)}+\norm{F}_{L^2(\Omega)}\right)^2.
\end{equation}
Therefore, the sequence $(\partial_\nu u_\ell)_{\ell \ge 1}$ is weakly convergent in $L^2(\Gamma)$, by Banach-Alaoglu's theorem. We denote by $w$ its weak limit in $L^2(\Gamma)$. On the other hand, $(u_\ell)_{\ell \ge 1}$ converges to $u$ in the norm-topology of $W^{2,p}(\Omega)$ according to {\it Step 2}, hence $(\partial_\nu u_\ell)_{\ell \ge 1}$ strongly converges to $\partial_\nu u$ in $L^p(\Gamma)$. Now, since $(\partial_\nu u_\ell)_{\ell\ \geq 1}$ converges to $w$ and to $\partial_\nu u$ in $D'(\Gamma)$, the space of distributions on $\Gamma$, we have $\partial_\nu u=w\in L^2(\Gamma)$ from the uniqueness of the limit, which proves the first claim of the result. Finally, \eqref{l3a} follows readily from \eqref{ess} and the weak convergence of $(\partial_\nu u_\ell)_{\ell \ge 1}$ to $\partial_\nu u$ in the Hilbert space $L^2(\Gamma)$.
\end{proof}

\begin{remark}
\label{rmk}
Proposition \ref{pr1} ensures us for all $q \in \cQ_c$, $c>0$, that all functions $u \in D(A_q)$ have a normal derivative $\partial_\nu u \in L^2(\Gamma)$ satisfying
\begin{equation}
\label{ess1} 
\norm{\partial_\nu u}_{L^2(\Gamma)}\leq C \norm{u}_{D(A_q)},
\end{equation}
where $C$ is a positive constant depending only on $\Omega$ and $q$. Here $\norm{\cdot}_{D(A_q)}$ denotes the usual operator norm  associated with $A_q$, i.e., $\norm{u}_{D(A_q)}=\norm{u}_{L^2(\Omega)}+\norm{A_q u}_{L^2(\Omega)}$ for all $u \in D(A_q)$.
\end{remark}

In particular, since all the eigenfunctions $\phi_k$, $k \geq 1$, lie in $D(A_q)$, we deduce from Remark \ref{rmk} that $\psi_k \in L^2(\Gamma)$ and that
\begin{equation}
\label{tm*} 
\norm{\psi_k}_{L^2(\Gamma)} \leq C(1+\abs{\lambda_k}).
\end{equation}
This nice features will prove to be useful for relating the normal derivative of the difference of two solutions to \eqref{1} associated with two arbitrary spectral parameters taken in the resolvent set of $A_q$, to the boundary spectral data of $A_q$. 

\begin{lemma}
\label{l5}
Let $q\in \cQ_c$, for some $c>0$. Pick $\lambda$ and $\mu$ in $\mathbb{C}\setminus \mathrm{Sp}(A_q)$, and for $f\in  H^{\frac{3}{2}}(\Gamma)$, denote by $u_\lambda$ (resp., $u_\mu$) the $H^1(\Omega)$-solution to \eqref{1} associated with $\lambda$ (resp., $\mu$). Then, we have
\begin{equation}
\label{l5a}
\partial_\nu (u_\lambda - u_\mu)=(\mu - \lambda) \sum_{k=1}^{+\infty}\frac{\langle f, \psi_k \rangle_{L^2(\Gamma)}}{(\lambda-\lambda_k)(\mu-\lambda_k)} \psi_k,
\end{equation}
the series being convergent in $L^2(\Gamma)$.
\end{lemma}
\begin{proof}  
Putting $v_{\lambda,\mu}:= u_\lambda-u_\mu$, it is easy to check that
\begin{equation}
\label{sys1}
 \left\{  
\begin{array}{ll}
 (-\Delta +q-\lambda)v_{\lambda,\mu}=(\lambda-\mu)u_{\mu} &\text{in}\ \Omega\\
   v_{\lambda,\mu}=0 &\text{on}\ \Gamma. \\          
 \end{array}                     
\right.
\end{equation}
Since $u_{\mu} \in L^2(\Omega)$ and $\lambda$ is in the resolvent set of $A_q$, we thus have 
$v_{\lambda,\mu}=(\lambda-\mu) (A_q-\lambda)^{-1} u_\mu$ and consequently
$$
v_{\lambda,\mu}=(\lambda-\mu) \sum_{k= 1}^{+\infty} \frac{\langle u_\mu , \phi_k\rangle_{L^2(\Omega)}}{\lambda_k-\lambda} \phi_k,
$$
where the series converges in $D(A_q)$. 
By continuity of the mapping $w \mapsto\partial_\nu w$ from $D(A_q)$ into $L^2(\Gamma)$, arising from \eqref{ess1}, the series
$\sum_{k= 1}^{+\infty} \frac{\langle u_\mu , \phi_k\rangle_{L^2(\Gamma)}}{\lambda_k-\lambda} \psi_k$ then converges in $L^2(\Gamma)$ and \begin{equation}
\label{sum}
\partial_\nu v_{\lambda,\mu}=(\lambda-\mu) \sum_{k= 1}^{+\infty} \frac{\langle u_\mu , \phi_k\rangle_{L^2(\Omega)}}{\lambda_k-\lambda} \psi_k.
\end{equation}
Further, by multiplying by $\overline{\phi_k}$ the first equation of \eqref{sys1} with $\lambda=\lambda_k$, integrating over $\Omega$ and applying the Green formula, we get in a similar fashion to \cite[Lemma 2.3]{KKS} or \cite[Lemma 2.1]{Ki1} that
$\langle u_\mu , \phi_k\rangle_{L^2(\Omega)}=-\frac{\left\langle f,\psi_k\right\rangle_{L^2(\Gamma)}}{\lambda_k-\mu}$. Finally,
\eqref{l5a} follows directly from this and \eqref{sum}.
\end{proof}

\section{Isozaki's asymptotic representation formula}
\label{sec-Isozaki}
In this section we aim to relate the Fourier transform of the difference $q_1-q_2$ of two potentials $q_j$, $j=1,2$, to the boundary spectral data of the Schr\"odinger operators $A_{q_j}$. This will be achieved by probing \eqref{1} with appropriately designed Dirichlet  
boundary data $f$ and collecting the Neumann response of the system. This idea, which is borrowed from the Born approximation method in scattering theory, was first applied to multidimensional inverse spectral analysis by Isozaki in \cite{I}.
This seminal article paved the way for numerous authors investigating inverse spectral problems (see e.g., \cite{BCFKS,CS,KKS,Ki1,P,S}) but in the context of this work, we shall essentially rely on \cite{P}, where Isozaki's approach was adapted to the framework of unbounded potentials.\\

\subsubsection{Test functions}
Let $\xi \in \mathbb{R}^n$. For all $\tau \ge \abs{\xi}$, we seek two test functions $f_\tau^\pm$ satisfying
\begin{equation}
\label{3.36}
( -\Delta -\lambda_{\tau}^\pm)f_\tau^\pm=0\ \mbox{in}\ \Omega,
\end{equation}
where $\lambda_\tau^\pm = (\tau \pm i)^2$, and such that
\begin{equation}
\label{2}
    \lim_{\tau \rightarrow +\infty} f_\tau^+(x) \overline{f_\tau^-(x)}=e^{-i\xi \cdot x},\; x\in \Omega,
\end{equation}
\begin{equation}
\label{3000}
    \sup_{\tau \ge \abs{\xi}} \norm{f_\tau^\pm}_{L^\infty(\Omega)} <\infty.
\end{equation}
For this purpose we pick $\eta \in \mathbb{S}^{n-1}$ such that $\xi \cdot \eta=0$ and for all $\tau  \ge \frac{\abs{\xi}}{2}$ we put
$$\beta_\tau = \sqrt{1-\frac{|\xi|^2}{4\tau^2}}\ \mbox{and}\ \eta_\tau^\pm = \beta_\tau \eta \mp \frac{\xi}{2\tau}, $$ 
in such a way that $\abs{\eta_\tau^\pm}=1$. Then, it is easy to check that the two functions
$$ f_\tau^\pm (x) := e^{i(\tau\pm i)\eta^{\pm}_{\tau} \cdot x},\; x\in \Omega,$$
satisfy the conditions \eqref{3.36} and \eqref{2}. Moreover, since 
$\abs{f_\tau^\pm (x)} \leq e^{\abs{x}}$ for all $x \in \overline{\Omega}$, we have
\begin{equation}
\label{3001}
\norm{f_\tau^\pm}_{L^r(X)} \leq \abs{X}^{1/r}  \sup_{x \in \overline{\Omega}} e^{\abs{x}},\ X=\Omega,\partial \Omega,
\end{equation}
whenever $r \in [2,+\infty)$ or $r=+\infty$. Notice that \eqref{3001} with $r=+\infty$ yields \eqref{3000}.

Let $q \in \cQ_c$. Then, for all $\tau \geq \abs{\xi}$ we have $q f_\tau^\pm \in L^2(\Omega)$ by \eqref{3001}, and the 
estimate
\begin{equation}
\label{tm1}
\norm{q f_\tau^\pm}_{L^2(\Omega)} \leq C \norm{q}_{L^{\max(2,3n \slash 5)}(\Omega)},
\end{equation}
where $C$ is a positive constant which is independent of $\tau$. As a matter of fact we have
$\norm{q f_\tau^\pm}_{L^2(\Omega)} \leq \norm{q}_{L^{2}(\Omega)} \norm{f_\tau^\pm}_{L^\infty(\Omega)} \leq C \norm{q}_{L^{2}(\Omega)}$
when $n=3$ and
$\norm{q f_\tau^\pm}_{L^2(\Omega)} \leq \norm{q}_{L^{3n \slash 5}(\Omega)} \norm{f_\tau^\pm}_{L^\frac{3n}{3n-6}(\Omega)} 
\leq C \norm{q}_{L^{3n \slash 5}(\Omega)}$
when $n \geq 4$.

\subsubsection{Probing the system with $f_\tau^\pm$}

For $j=1,2$, let $q_j \in \cQ_c$ and let $z\in \mathbb{C}\setminus \mathrm{Sp}(A_{q_j})$. We denote by $u_{j,z}^{\pm}$ the 
$W^{2,p}(\Omega)$-solution to the BVP
\begin{equation}
\label{tm0}
 \left\{
\begin{array}{ll}
( -\Delta +{{q_j}}-z)u = 0  &\text{in}\ \Omega,\\
u = f_{\tau}^\pm & \text{on}\ \Gamma.
 \end{array}
\right. 
\end{equation}
Since $( -\Delta +{{q_j}}-z)f^\pm_\tau =({{q_j}}+\lambda_{\tau}^\pm -z) f^\pm_\tau$ from \eqref{3.36}, the function
\begin{equation}
\label{vz1}
v_{j,z}^\pm := u_{j,z}^\pm-f_\tau^\pm
\end{equation}
 then solves
 $$
 \left\{
\begin{array}{ll}
( -\Delta +{{q_j}}-z) v = -( -\Delta +{{q_j}}-z)f^\pm_\tau  & \text{in}\ \Omega\\
v= 0 & \text{on}\ \Gamma,
 \end{array}
\right. 
$$
which amounts to saying that
\begin{equation}
\label{vz2}
v^\pm_{j,z}=-(A_{q_j}-z)^{-1}(q_j+\lambda^\pm_\tau-z)f_\tau^\pm. 
\end{equation}
In the special case where $z=\lambda_\tau^\pm$, the above identity reads $v_{j,\lambda_\tau^\pm}^\pm = -(A_{q_j}-\lambda_{\tau}^\pm)^{-1} (q_j f_\tau^\pm)$.
Since $\im\ \lambda_\tau^\pm =\pm 2 \tau$, we deduce from \eqref{tm1} that
\begin{equation}
\label{vz3}
\norm{v^\pm_{j,\lambda_\tau^\pm}}_{L^2(\Omega)} \leq C \norm{q_j}_{L^{\max(2,3n \slash 5)}(\Omega)} \tau^{-1},\ \tau \ge \abs{\xi},
\end{equation}
where the constant $C>0$ is independent of $\tau$. From this and \eqref{tm1} it then follows that 
$\abs{\int_\Omega v_{j,\lambda_\tau^+}^+ q_j f_\tau^-dx} \leq C^2 \norm{q_j}_{L^{\max(2,3n \slash 5)}(\Omega)}^2 \tau^{-1}$, which yields that
\begin{equation}
\label{l4a}
\lim_{\tau\to+\infty} \int_\Omega v_{j,\lambda_\tau^+}^+ q_j f_\tau^-dx=0.
\end{equation}
Armed with \eqref{l4a} we turn now to establish the Isozaki formula for the unbounded potentials $q_j$, $j=1,2$. 

\subsubsection{Isozaki's asymptotic formula}
For $\tau \geq \abs{\xi}$, we introduce
\begin{equation}
\label{2.16}
S_{j,\tau} := \langle \partial_\nu u_{j,\lambda_\tau^+}^+ ,f_\tau^- \rangle_{L^2(\Gamma)},\ j=1,2,
\end{equation}
where we recall from \eqref{vz1}-\eqref{vz2} that $u_{j,\lambda_\tau^+}^+=v_{j,\lambda_\tau^+}^+ + f_\tau^+$ and $v_{j,\lambda_\tau^+}^+=-(A_{q_j}-\lambda_\tau^+)^{-1} (q_j f_\tau^+)$.
Notice that since $v_{j,\lambda_\tau^+}^+ \in D(A_{q_j})$, we have $\partial_\nu u_{j,\lambda_\tau^+}^+ \in L^2(\Gamma)$ from Remark \ref{rmk}, and hence $S_{j,\tau}$ is well-defined.

Having seen this, we can extend the classical Isozaki formula to the case of unbounded potentials.

\begin{proposition}
\label{p1}
For $c>0$ fixed, let $q_j \in \cQ_c$, $j=1,2$. Then, for all $\xi \in \mathbb{R}^n$, it holds true that
\begin{equation}
\label{p1a}
\lim_{\tau \rightarrow +\infty} (S_{1,\tau}-S_{2,\tau}) =\int_\Omega (q_1-q_2) e^{-i\xi\cdot x} dx.
\end{equation}
\end{proposition}
\begin{proof}
Bearing in mind that
\begin{equation} 
\label{20}
 \left\{
\begin{array}{ll}
( -\Delta +{{q_j}}-\lambda_\tau^+)u_{j,\lambda_\tau^+}^+ = 0  & \text{in}\ \Omega\\
u_{j,\lambda_\tau^+}^+ = f_{\tau}^+ & \text{on}\ \Gamma,
 \end{array}
\right. 
\end{equation}
we multiply the first line of \eqref{20} by $\overline{f_\tau^-}$ and integrate on $\Omega$. Applying the Green formula, we obtain that
\begin{eqnarray*}
0 &= & \int_\Omega (-\Delta+q_j-\lambda_\tau^+)u^+_{j,\lambda_\tau^+}\overline{f_\tau^- (x)} dx \\
& = & \int_\Gamma f_\tau^+ \overline{ \partial_\nu f_\tau^-} d\sigma-\int_\Gamma (\partial_\nu u_{j,\lambda_{\tau}^+}^+) \overline{ f_\tau^-} d\sigma + \int_\Omega u_{j,\lambda_{\tau}^+}^+ \overline{ (-\Delta+q_j-\lambda_\tau^-)f_\tau^-} dx \\
& = & \int_\Gamma f_\tau^+ \overline{ \partial_\nu f_\tau^-} d\sigma - S_{j,\tau} + \int_\Omega u_{j,\lambda_{\tau}^+}^+ q_j \overline{f_\tau^-} dx,
\end{eqnarray*}
where we used \eqref{3.36} and \eqref{2.16} in the last line.
As a consequence we have
$$
S_{1,\tau}-S_{2,\tau}=\int_\Omega \big(q_1u_{1,\lambda_{\tau}^+}^+ - q_2 u_{2,\lambda_{\tau}^+}^+\big)\overline{f_\tau^-} dx,
$$
and hence
$$
S_{1,\tau} -S_{2,\tau} = \int_\Omega (q_1-q_2) f_\tau^+ \overline{f_\tau^-} dx + \int_\Omega q_1 v_{1,\lambda_{\tau}^+}^+ \overline{f_\tau^-} dx - \int_\Omega q_2 
v_{2,\lambda_{\tau}^+}^+ \overline{f_\tau^-} dx,
$$
from the identities $u_{j,\lambda_\tau^+}^+=v_{j,\lambda_\tau^+}^+ + f^+_\tau$, for $j=1,2$.
Taking the limit as $\tau \to +\infty$ in the above line then yields
\begin{equation}
\label{tm3}
\lim_{\tau\to+\infty} \left( S_{1,\tau}-S_{2,\tau} - \int_\Omega (q_1-q_2) f_\tau^+ \overline{f_\tau^-} dx \right)=0,
\end{equation}
with the aid of \eqref{l4a}. Finally, since $q_1-q_2 \in L^1(\Omega)$, we have 
$$\lim_{\tau\to+\infty} \int_\Omega (q_1-q_2) f_\tau^+ \overline{f_\tau^-} dx=\int_\Omega (q_1-q_2) e^{-i \xi \cdot x} dx$$ 
from \eqref{2} and the dominated convergence theorem, and the desired result follows from this and \eqref{tm3}.                        
\end{proof}

\section{Proof of Theorem \ref{t1}} 
\label{sec-proof}
We use the same notations as in Section \ref{sec-Isozaki}. Namely, for $z \in \mathbb{C}\setminus \mathrm{Sp}(A_{q_j})$, $j=1,2$, we consider 
the $W^{2,p}(\Omega)$-solution $u^+_{j,z}$ to the BVP \eqref{tm0}. Since $q_j f_\tau \in L^2(\Omega)$ according to  \eqref{tm1}, we have
$u_{j,z}^+-f_\tau^+ \in D(A_{q_j})$ from \eqref{vz1}-\eqref{vz2}, and hence $\partial_\nu u_{j,z}^+ \in L^2(\Gamma)$ by Remark \ref{rmk}. 
Therefore, for all $\mu \in \mathbb{C} \setminus \left( \mathrm{Sp}(A_{q_1}) \cup \mathrm{Sp}(A_{q_2}) \right)$ the normal derivative
of $v^+_{j,\lambda_\tau^+,\mu} := u^+_{j,\lambda_\tau^+}-u^+_{j,\mu}$ lies in $L^2(\Gamma)$ and we have
\begin{eqnarray}
S_{1,\tau}-S_{2,\tau} 
&= & \langle \partial_\nu u^+_{1,\lambda^+_\tau} -\partial_\nu u^+_{2,\lambda^+_\tau},f_\tau^- \rangle_{L^2(\Gamma)} \nonumber\\
& = & \langle \partial_\nu v^+_{1,\lambda^+_\tau,\mu} ,f_\tau^- \rangle_{L^2(\Gamma)}-\langle \partial_\nu v^+_{2,\lambda^+_\tau,\mu},f_\tau^- \rangle_{L^2(\Gamma)}+\langle \partial_\nu u^+_{1,\mu} -\partial_\nu u^+_{2,\mu},f_\tau^- \rangle_{L^2(\Gamma)}, \label{tm5}
\end{eqnarray}
according to \eqref{2.16}. 
Let us examine the last term on the right-hand-side of \eqref{tm5}. Applying H\"older's inequality, we obtain that
$$ \abs{\langle \partial_\nu u^+_{1,\mu} -\partial_\nu u^+_{2,\mu},f_\tau^- \rangle_{L^2(\Gamma)}} \leq \norm{\partial_\nu u^+_{1,\mu} -\partial_\nu u^+_{2,\mu}}_{L^p(\Gamma)} \norm{f_\tau^-}_{L^{p^\prime}(\Gamma)},$$
where $p^\prime=\frac{2n}{n-2}$ is the H\"older conjugate of $p=\frac{2n}{n+2}$.
Thus, we have
$\lim_{\mu \to -\infty} \langle \partial_\nu u^+_{1,\mu} -\partial_\nu u^+_{2,\mu},f_\tau^- \rangle_{L^2(\Gamma)} = 0$ by Lemma \ref{lemma 2.2}, and hence
\begin{equation}
\label{t1c}
S_{1,\tau}-S_{2,\tau}  = 
\lim_{\mu \rightarrow -\infty} \langle \partial_\nu v^+_{1,\lambda^+_\tau,\mu}-\partial_\nu v^+_{2,\lambda^+_\tau,\mu} ,f_\tau^- \rangle_{L^2(\Gamma)},
\end{equation}
from \eqref{tm5}.

Further, applying Lemma \ref{l5} we get through direct computation that 
\begin{equation}
\label{tm6}
\langle \partial_\nu v^+_{1,\lambda^+_\tau,\mu}-\partial_\nu v^+_{2,\lambda^+_\tau,\mu} ,f_\tau^- \rangle_{L^2(\Gamma)} 
=\sum_{k = 1}^\infty \left( A_k(\mu,\tau)+B_k(\mu,\tau)+C_k(\mu,\tau) \right),
\end{equation}
where
$$A_k(\mu,\tau):=\frac{\mu-\lambda^+_\tau}{(\lambda_\tau^+-\lambda_{1,k})(\mu - \lambda_{1,k})}
\langle f_\tau^+, \psi_{1,k}-\psi_{2,k}\rangle_{L^2(\Gamma)} \overline{\langle f^-_\tau ,\psi_{1,k} \rangle}_{L^2(\Gamma)},$$
$$B_k(\mu,\tau):=\frac{\mu-\lambda^+_\tau}{(\lambda_\tau^+-\lambda_{1,k})(\mu - \lambda_{1,k})} 
\langle f_\tau^+, \psi_{2,k}\rangle_{L^2(\Gamma)} \overline{\langle f^-_\tau ,\psi_{1,k}-\psi_{2,k}\rangle}_{L^2(\Gamma)}$$
and
$$C_k(\mu,\tau):=\left(\frac{\mu-\lambda^+_\tau}{(\lambda_\tau^+-\lambda_{1,k})(\mu - \lambda_{1,k})}-\frac{\mu-\lambda^+_\tau}{(\lambda_\tau^+-\lambda_{2,k})(\mu - \lambda_{2,k})}\right)\langle f_\tau^+, \psi_{2,k}\rangle_{L^2(\Gamma)}\overline{\langle f^-_\tau ,\psi_{1,k} \rangle}_{L^2(\Gamma)}.$$
Let us first examine $A_k(\mu,\tau)$ and $B_k(\mu,\tau)$. For this purpose we recall from \eqref{tm*} that the estimate
$\norm{\psi_{j,k}}_{L^2(\Gamma)}\leq C \left(1+\abs{\lambda_{j,k}} \right)$ holds for $j=1,2$ and all $k \geq 1$, where $C>0$ denotes a generic
constant which depends only on $\Omega$ and $q_j$.
This and \eqref{3001} then yield that
\begin{equation}
\label{tm7}
\abs{\langle f^\pm_\tau ,\psi_{j,k} \rangle_{L^2(\Gamma)}} \leq C \left(1+\abs{\lambda_{j,k}} \right),\ k \geq 1.
\end{equation}
Moreover, since
$\sup_{k\geq1} \abs{\lambda_{1,k}-\lambda_{2,k}}<\infty$ by \eqref{t1a}, we have
$$
\abs{\lambda_{2,k}} \leq C(1+\abs{\lambda_{1,k}}),\ k \geq 1. 
$$
From this, \eqref{3001} and \eqref{tm7}, it then follows for all $\mu \leq -(1+c)$ and all $\tau \geq 1 + \abs{\xi}$, that
\begin{equation}
\label{tm8}
\abs{A_k(\mu,\tau)} + \abs{B_k(\mu,\tau)} \leq C_\tau \norm{\psi_{1,k}-\psi_{2,k}}_{L^2(\Gamma)},\ k \geq 1.
\end{equation}
Here and below, $C_\tau$ is a positive constant independent of $k$ and $\mu$, which possibly depends on $\tau$ and may change from line to line.

Similarly, by rewriting
$\frac{\mu-\lambda^+_\tau}{(\lambda_\tau^+-\lambda_{1,k})(\mu - \lambda_{1,k})}-
\frac{\mu-\lambda^+_\tau}{(\lambda_\tau^+-\lambda_{2,k})(\mu - \lambda_{2,k})}$ as
$\frac{\lambda_{1,k}-\lambda_{2,k}}{(\lambda_\tau^+ - \lambda_{1,k})(\lambda_\tau^+ - \lambda_{2,k})}
- \frac{\lambda_{1,k}-\lambda_{2,k}}{(\mu - \lambda_{1,k})(\mu - \lambda_{2,k})}$
and using that $\sup_{k\geq1} \abs{\lambda_{1,k}-\lambda_{2,k}}<\infty$, we obtain for all $\mu \leq -(1+c)$ and $\tau \geq 1 + \abs{\xi}$ that
\begin{equation}
\label{tm9}
\abs{C_k(\mu,\tau)} \leq C_\tau \left| \frac{\langle f^-_\tau ,\psi_{1,k} \rangle_{L^2(\Gamma)}}{\lambda_\tau^+-\lambda_{1,k}} \right| \left| \frac{\langle f^+_\tau ,\psi_{2,k} \rangle_{L^2(\Gamma)}}{\lambda_\tau^--\lambda_{2,k}} \right|,\ k \geq 1.
\end{equation}
Moreover, remembering that $u_{j,\lambda_\tau^\pm}^\pm$ is the solution $u_{j,z}^\pm$ to \eqref{tm0} with $z=\lambda_\tau^\pm$, we find upon multiplying the first 
line of the corresponding BVP by $\overline{\phi_{j,k}}$, integrating over $\Omega$ and then applying the Green formula, that
$$\langle u_{j,\lambda_\tau^\pm}^\pm , \phi_{j,k} \rangle_{L^2(\Omega)}
=-\frac{\langle f_\tau^\pm, \psi_{j,k} \rangle_{L^2(\Gamma)}}{\lambda_\tau^\pm - \lambda_{j,k}},\ k \geq 1.$$
Parseval's formula thus yields
\begin{equation}
\label{tt1}
\sum_{k=1}^{+\infty} \left| \frac{\langle f^\pm_\tau ,\psi_{j,k} \rangle_{L^2(\Gamma)}}{\lambda_\tau^\pm-\lambda_{j,k}} \right|^2=
\norm{u_{j,\lambda_\tau^\pm}^\pm}_{L^2(\Omega)}^2<\infty,\ j=1,2.
\end{equation}
Putting this together with \eqref{tm9}, and combining \eqref{tm8} with \eqref{t1a}, we infer from \eqref{tm6} and the dominated convergence theorem that
\begin{equation}
\label{t1g}
S_{1,\tau}-S_{2,\tau}=\lim_{\mu \rightarrow -\infty} \langle \partial_\nu v^+_{1,\lambda^+_\tau,\mu}-\partial_\nu v^+_{2,\lambda^+_\tau,\mu} ,f_\tau^- \rangle_{L^2(\Gamma)}
=\sum_{k = 1}^{+\infty} \left( A_k^*(\tau)+B_k^*(\tau)+ C_k^* (\tau) \right),
\end{equation}
where
\begin{equation}
\label{tm10}
A_k^*(\tau):=\frac{1}{\lambda_\tau^+-\lambda_{1,k}} 
\langle f_\tau^+, \psi_{1,k}-\psi_{2,k}\rangle_{L^2(\Gamma)} \overline{\langle f^-_\tau ,\psi_{1,k} \rangle}_{L^2(\Gamma)},
\end{equation}
\begin{equation}
\label{tm11}
B_k^*(\tau):=\frac{1}{\lambda_\tau^+-\lambda_{1,k}} \langle f_\tau^+, \psi_{2,k}\rangle_{L^2(\Gamma)} \overline{\langle f^-_\tau ,\psi_{1,k}-\psi_{2,k}\rangle}_{L^2(\Gamma)}
\end{equation}
and
\begin{equation}
\label{tm12}
C_k^*(\tau):=\frac{\lambda_{1,k}-\lambda_{2,k}}{(\lambda_\tau^+-\lambda_{1,k})(\lambda_\tau^+-\lambda_{2,k})}
\langle f_\tau^+, \psi_{2,k} \rangle_{L^2(\Gamma)} \overline{\langle f^-_\tau ,\psi_{1,k} \rangle_{L^2(\Gamma)}}.
\end{equation}
Further, since $\im (\lambda_\tau^+ - \lambda_{j,k})=2\tau$ for all $k \ge 1$, it follows readily from \eqref{3001} that
$$
\abs{A_k^*(\tau)} + \abs{B_k^*(\tau)} \leq C \tau^{-1} \norm{\psi_{1,k}-\psi_{2,k}}_{L^2(\Gamma)} \left( \norm{\psi_{1,k}}_{L^2(\Gamma)} + \norm{\psi_{2,k}}_{L^2(\Gamma)} \right)
$$
and
$$
\abs{C_k^*(\tau)} \leq C \tau^{-2} \abs{\lambda_{1,k}-\lambda_{2,k}}  \norm{\psi_{1,k}}_{L^2(\Gamma)} \norm{\psi_{2,k}}_{L^2(\Gamma)},
$$
where $C$ is a positive constant which is independent of $k$ and $\tau$.
As a consequence we have
$$
\lim_{\tau\to+\infty}A_k^*(\tau)=\lim_{\tau\to+\infty}B_k^*(\tau)=\lim_{\tau\to+\infty}C_k^*(\tau)=0,\ k \ge 1,
$$
which together with \eqref{t1g} yields for any natural number $N$, that
\begin{equation}
\label{t1h}
\lim_{\tau\to+\infty} \abs{S_{1,\tau}-S_{2,\tau}} \leq \limsup_{\tau\to+\infty} \sum_{k = N}^{+\infty} \left( \abs{A_k^*(\tau)}+ \abs{B_k^*(\tau)}+ \abs{C_k^*(\tau)} \right).
\end{equation}
On the other hand, applying the Cauchy-Schwarz inequality in \eqref{tm10}, we get that
\begin{eqnarray}
\sum_{k = N}^{+\infty} 
\abs{A_k^*(\tau)} & \leq & \norm{f_\tau^+}_{L^2(\Omega)} \left( \sum_{k=1}^{+\infty} \left| \frac{\langle f^-_\tau ,\psi_{1,k} \rangle_{L^2(\Gamma)}}{\lambda_\tau^+-\lambda_{1,k}} \right|^2 \right)^{\frac{1}{2}} \left(\sum_{k=N}^{+\infty} \norm{\psi_{1,k}-\psi_{2,k}}_{L^2(\Gamma)}^2\right)^{\frac{1}{2}} \nonumber \\
&\leq &  \norm{f_\tau^+}_{L^2(\Omega)} \norm{u_{1,\lambda_\tau^+}^-}_{L^2(\Omega)}
\left(\sum_{k=N}^{+\infty} \norm{\psi_{1,k}-\psi_{2,k}}_{L^2(\Gamma)}^2\right)^{\frac{1}{2}},
\label{tm13}
\end{eqnarray}
from \eqref{tt1}. Next, since $\sup_{\tau \ge 1} \norm{f_\tau^+}_{L^2(\Omega)} \norm{u_{1,\lambda_\tau^+}^-}_{L^2(\Omega)}<\infty$ by virtue of
 \eqref{3001}, \eqref{vz1} and \eqref{vz3}, it follows from \eqref{tm13} that
\begin{equation}
\label{t1i}
\limsup_{\tau\to+\infty}\sum_{k = N}^{+\infty} \abs{A_k^*(\tau)} \leq C \left( \sum_{k=N}^\infty\norm{\psi_{1,k}-\psi_{2,k}}_{L^2(\Gamma)}^2 \right)^{\frac{1}{2}}
\end{equation}
for some constant $C>0$ which is independent of $N$.

Arguing as before with \eqref{tm11} and \eqref{tm12} instead of \eqref{tm10}, we find that 
$$
\limsup_{\tau\to+\infty}\sum_{k = N}^\infty \abs{B_k^*(\tau)}\leq C \left( \sum_{k=N}^{+\infty}\norm{\psi_{1,k}-\psi_{2,k}}_{L^2(\Gamma)}^2 \right)^{\frac{1}{2}}
$$
and that
$$
\limsup_{\tau\to+\infty}\sum_{k = N}^{+\infty} \abs{C_k^*(\tau)} \leq C\sup_{k\geq N}\abs{\lambda_{1,k}-\lambda_{2,k}},
$$
which together with \eqref{t1h} and \eqref{t1i} yields
\begin{equation}
\label{limsup}
\limsup_{\tau\to+\infty} \abs{S_{1,\tau}-S_{2,\tau}}\leq C\left(\sup_{k\geq N}\abs{\lambda_{1,k}-\lambda_{2,k}}+\left( \sum_{k=N}^{+\infty}\norm{\psi_{1,k}-\psi_{2,k}}_{L^2(\Gamma)}^2 \right)^\frac{1}{2} \right),
\end{equation}
where $C$ still denotes a positive constant independent of $N$. With reference to \eqref{t1a}, we thus get that $\lim_{\tau\to+\infty} \abs{S_{1,\tau}-S_{2,\tau}}=0$ by sending $N\to+\infty$ in \eqref{limsup}. This and \eqref{p1a} entail that 
\begin{equation}
\label{tm14}
\int_\Omega e^{-ix\cdot\xi}(q_1-q_2)dx=0.
\end{equation}
Finally, since $q_1-q_2\in L^1(\Omega)$ and since \eqref{tm14} holds for all $\xi\in\R^n$, we deduce from the injectivity of the Fourier transform that $q_1=q_2$. This terminates the proof of Theorem \ref{t1}.

\appendix

\section{The perturbed Dirichlet-Laplacian $-\Delta +q$ in $L^2(\Omega)$}
\label{app-A}
For $q \in \cQ_c$, $c>0$, let $a_q$ be the Hermitian form defined in \eqref{def-aq}.
Since $H^1(\Omega)$ is continuously embedded in $L^{\frac{2n}{n-2}}(\Omega)$ from the Sobolev embedding theorem,
we have $u \overline{v} \in L^{\frac{n}{n-2}}(\Omega) \subset L^{\frac{3n}{3n-5}}(\Omega)$ for all $(u,v) \in D(a_q) \times D(a_q)$, and
$$ \norm{u \overline{v}}_{L^{\frac{3n}{3n-5}}(\Omega)} \le C \norm{u \overline{v}}_{L^{\frac{n}{n-2}}(\Omega)} \leq C \norm{u}_{L^{\frac{2n}{n-2}}(\Omega)}
\norm{v}_{L^{\frac{2n}{n-2}}(\Omega)} \leq C  \norm{u}_{H^1(\Omega)} \norm{v}_{H^1(\Omega)}. $$
Here and below, $C$ denotes a generic positive constant which depends only on $n$ and $\Omega$.
Thus, $q u \overline{v} \in L^1(\Omega)$ and 
$$ \norm{q u \overline{v}}_{L^1(\Omega)} \leq \norm{q}_{L^\frac{3n}{5}(\Omega)}  \norm{u \overline{v}}_{L^{\frac{3n}{3n-5}}(\Omega)} \leq C \norm{q}_{L^\frac{3n}{5}(\Omega)}  \norm{u}_{H^1(\Omega)} \norm{v}_{H^1(\Omega)}, $$
by H\" older's inequality, showing that $a_q$ is continuous on $D(a_q) \times D(a_q)$:
$$ \abs{a_q(u,v)} \le \norm{\nabla u}_{L^2(\Omega)} \norm{\nabla v}_{L^2(\Omega)} + \norm{q u \overline{v}}_{L^1(\Omega)}
\le \left(1+C \norm{q}_{L^\frac{3n}{5}(\Omega)} \right) \norm{u}_{H^1(\Omega)} \norm{v}_{H^1(\Omega)}. $$
Denote by $A_q$ the unbounded operator generated by $a_q$, acting in $L^2(\Omega)$ on its domain $D(A_q) := \{ u \in H_0^1(\Omega),\ (-\Delta+q) u \in L^2(\Omega) \}$. 
Since 
$$a_q(u,u) + c \norm{u}_{L^2(\Omega)}^2 = \norm{\nabla u}_{L^2(\Omega)} \geq C \norm{u}_{H^1(\Omega)},\ u \in D(a_q), $$
by the Poincar\'e inequality, $D(A_q)$ is dense in $L^2(\Omega)$ and the operator $A_q$ is self-adjoint in $L^2(\Omega)$. 

\bigskip

\paragraph{\bf Acknowledgement}
Two of the authors, Y. Kian and \'E. Soccorsi, are partially supported by the Agence Nationale de la Recherche under grant ANR-17- CE40-0029.


\begin{thebibliography}{99}
\bibitem{AS} {\sc G. Alessandrini and  J. Sylvester}, {\em Stability for multidimensional inverse spectral problem}, Commun. Partial Diff. Eqns., \textbf{15} (5) (1990), 711-736.
 \bibitem{A}{\sc  V. A. Ambartsumian}, {\em \"Uber eine Frage der Eigenwerttheorie}, Z. Phys., \textbf{53} (1929), 690-695.
\bibitem{Bo}{\sc G. Borg}, {\em Eine Umkehrung der Sturm-Liouvilleschen Eigenwertaufgabe}, Acta Math.,
\textbf{78} (1946), 1-96. 
\bibitem{Bel87}{\sc M. Belishev}, {\em An approach to multidimensional inverse problems for the wave equation}, Dokl. Akad. Nauk SSSR, \textbf{297} (1987), 524-527.
\bibitem{Bel92} {\sc M. Belishev and Y. Kurylev}, {\em To the reconstruction of a Riemannian manifold via its spectral data (BC-method)}, Comm. Partial Differential Equations, \textbf{17} (1992), 767-804.
\bibitem{BCFKS}{\sc M. Bellassoued, M. Choulli, D. D. Feirrera, Y. Kian, P. Stefanov}, {\em A Borg-Levinson theorem for magnetic Schr\"odinger operators on a Riemannian manifold}, to appear in Annales de l'Institut Fourier, arXiv:1807.08857.
\bibitem{BCY}{\sc M. Bellassoued, M. Choulli, M. Yamamoto}, {\em Stability estimate for an inverse wave equation and a multidimensional Borg-Levinson theorem}, J. Diff. Equat., \textbf{247}(2)  (2009), 465-494.
\bibitem{BD} {\sc M. Bellassoued and D. Dos Santos Ferreira}, {\em Stability estimates for the anisotropic wave equation from the Dirichlet to-Neumann map,} Inverse Problems and Imaging, \textbf{5} No 4 (2011), 745-773.
\bibitem{BM} {\sc M. Bellassoued and Y. Mannoubi}, {\em A partial data inverse problem for the electro-magnetic wave equation and application to the related Borg-Levinson Theorem,} Inverse Problems 36 (2020), 115001.
\bibitem{CK1} {\sc B. Canuto and O. Kavian}, {\em Determining Coefficients in a Class of Heat Equations via Boundary Measurements}, SIAM Journal on Mathematical Analysis, \textbf{32} no. 5 (2001), 963-986.
\bibitem{CK2} {\sc B. Canuto and O. Kavian}, {\em Determining Two Coefficients in Elliptic Operators via Boundary Spectral Data: a Uniqueness Result}, Bolletino Unione Mat. Ital. Sez. B Artic. Ric. Mat. (8), \textbf{7} no. 1 (2004), 207-230.
\bibitem{Ch}{\sc M. Choulli}, {\em Une introduction aux probl\`emes inverses elliptiques et paraboliques}, Math\'ematiques et Applications, Vol. 65, Springer-Verlag, Berlin, 2009.
\bibitem{CS}{\sc M. Choulli and P. Stefanov}, {\em Stability for the multi-dimensional Borg-Levinson theorem with partial spectral data}, 
Commun. Partial Diff. Eqns., \textbf{38} (3) (2013), 455-476.
\bibitem{GL} {\sc I. M. Gel'fand and B. M. Levitan}, {\em On the determination of a differential equation from its spectral function}, Izv. Akad. Nauk USSR, Ser. Mat., \textbf{15} (1951), 309-360.
\bibitem{Gr}  Grisvard. P., Elliptic Problems in Nonsmooth Domains, Pitman, 1985.
\bibitem{I} {\sc H. Isozaki}, {\em Some remarks on the multi-dimensional Borg-Levinson theorem},
J. Math. Kyoto Univ., \textbf{31} (3) (1991), 743-753.
\bibitem{KK} {\sc A. Katchalov and Y. Kurylev}, {\em Multidimensional inverse problem with incomplete boundary spectral data},
Commun. Partial Diff. Eqns., \textbf{23} (1998),  55-95. 
\bibitem{KKL} {\sc A. Katchalov,  Y. Kurylev, M.  Lassas}, {\em Inverse boundary spectral problems}, Chapman \& Hall/CRC, Boca Raton, FL, 2001, 123, xx+290.
\bibitem{KKLM} {\sc A. Katchalov, Y. Kurylev, M. Lassas, N. Mandache}, {\em Equivalence of time-domain inverse problems and boundary spectral problem}, Inverse Probl., \textbf{20} (2004), 419-436.
\bibitem{KKS}{\sc O. Kavian, Y. Kian, E. Soccorsi}, {\it Uniqueness and stability results for an inverse spectral problem in a periodic waveguide}, J. Math. Pures Appl., \textbf{104} (2015), no. 6, 1160-1189.
\bibitem{Ki1}{\sc Y. Kian}, {\em  A multidimensional Borg-Levinson theorem for magnetic Schr\"odinger operators with partial spectral data}, Journal of Spectral Theory, \textbf{8} (2018), 235-269.
\bibitem{Ki2}{\sc Y. Kian}, {\em Simultaneous determination of coefficients, internal sources and an obstacle of a diffusion equation from a single measurement}, preprint,  	arXiv:2007.08947.
\bibitem{KLLY}{\sc Y. Kian, Z. Li, Y. Liu, M. Yamamoto}, {\em The uniqueness of inverse problems for a fractional equation with a single measurement}, to appear in Math. Ann.,  https://doi.org/10.1007/s00208-020-02027-z.
\bibitem{KOM}{\sc Y. Kian, L. Oksanen, M. Morancey}, {\em Application of the boundary control method to partial data Borg-Levinson inverse spectral problem}, , MCRF, \textbf{9} (2019),  289-312.
\bibitem{Lassas2010} {\sc M. Lassas and L.Oksanen},  {\em An inverse problem for a wave equation with sources and observations on disjoint sets}, Inverse Problems,  \textbf{26} (2010), 085012.
\bibitem{L}{\sc N. Levinson}, {\em The inverse Sturm-Liouville problem}, Mat. Tidsskr. B., \textbf{1949} (1949), 25-30.
\bibitem{LM1}{\sc J.-L. Lions and E. Magenes}, 
{\em Non-homogeneous Boundary Value Problems and Applications}, Vol. I, 
Spring
er-Verlag, Berlin, 1972.
\bibitem{NSU} {\sc A. Nachman, J. Sylvester, G. Uhlmann}, {\em An n-dimensional Borg-Levinson theorem}, Comm. Math. Phys., \textbf{115} (4) (1988), 595-605.
\bibitem{PS} {\sc L. P\"aiv\"arinta, V. Serov}, {\em An n-dimensional Borg-Levinson theorem for singular
potentials}, Adv. in Appl. Math., \textbf{29} (2002), no. 4, 509-520.
\bibitem{P} {\sc V. Pohjola}, \emph{Multidimensional Borg-Levinson theorems for unbounded potentials}, Asymptot. Anal., \textbf{110} (2018), no. 3-4, 203-226.
\bibitem{S} {\sc \'E. Soccorsi}, \emph{Multidimensional Borg-Levinson inverse spectral problems}, Contemp. Math., \textbf{757} (2020), 19-50.
\end{thebibliography}
\end{document}